\documentclass[preprint,12pt,times]{article}
\usepackage{graphicx}
\usepackage{amsthm}
\usepackage{amssymb}
\usepackage{amsmath}
\usepackage{amscd}
\usepackage{bbm}
\usepackage{float}
\usepackage{booktabs}
\numberwithin{equation}{section}

\newtheorem{theorem}{Theorem}[section]
\newtheorem{lemma}{Lemma}[section]
\newtheorem{remark}{Remark}[section]

\usepackage[numbers,sort&compress]{natbib}

\numberwithin{figure}{section}
\numberwithin{table}{section}

\usepackage{pgfplots}
\newcommand\btd{\raise 2pt \hbox{$\hat\bigtriangledown$}\hskip 1.5pt}
\newcommand\bt{\raise 2pt \hbox{$\bigtriangledown$}\hskip 1.5pt}

\begin{document}
\date{}
\title{Kinetic energy and streamline properties for irrotational equatorial wind waves}
\author{Jian Li, ~Shaojie Yang\thanks{Corresponding author: jianli\_jakura@163.com(Jian Li); shaojieyang@kust.edu.cn (Shaojie Yang)} \\~\\
\small Department of  Mathematics,~~Kunming University of Science and Technology,  \\
\small Kunming, Yunnan 650500, China}

\date{}
\maketitle
\begin{abstract}
~\\In this paper, we investigate kinetic energy and streamline properties  for an irrotational periodic geophysical traveling surface water waves  propagating in equatorial oceanic regions. Relying on the methods from  complex analysis, we prove the logarithmic convexity and monotonicity of specific flow variables. By means of conformal mappings, we derive some qualitative results for kinetic energy and streamline, such as streamline time-period being independent of any moment and any point on the streamline in steady flow, the concavity and monotonicity of total kinetic energy within the region between two streamlines and the convexity and monotonicity of total kinetic energy over a streamline time-period. Moreover, we present several results about irrotational equatorial wind waves, such as an upper bound of the minimum of streamline time-period, an upper bound of the maximum of area within the region between two streamlines. Taking advantage of the Bernoulli's law and the Schwarz reflection principle, we show that the extremum of the kinetic energy is attained on the free surface for irrotational equatorial wind waves.\\


\end{abstract}
\noindent\rule{16.5cm}{0.5pt}

\section{Introduction}
Equatorial  waves exhibit particular dynamics due to the vanishing of the Coriolis parameter along the Equator.
The equatorial region throughout the extent of the Pacific Ocean (about $13,000~km$) is the Equatorial Undercurrent (EUC)\cite{r1}. Owing to equatorial winds that blow westward, the surface water flow is directed westward \cite{r2}, and the flow reverses at a depth of several tens of meters \cite{r3}. The Pacific EUC is confined to a shallow surface layer which is less than $200~m$ beneath the surface \cite{r4}, and which is symmetric about the Equator, it is typically about $300~km$ in width and symmetric about the Equator, with a maximum speed of about $1~ m/s$ \cite{r5}.
Equatorial  waves are highly fascinating from both physical and mathematical perspectives. Since the Coriolis forces vanish at the Equator, one can  apply the $f$-plane approximation to the full geophysical governing equations in the equatorial region \cite{gallagher2007influence,r2, constantin2012modelling}. Furthermore, since the equator acts as a natural waveguide, equatorial waves are predominantly zonal, allowing us to approximate the flows as being two-dimensional. Recently, it has been shown that a number of qualitative analytical features, which apply to gravity waves with vorticity, similarly apply to equatorial waves\cite{fan2019mean,constantin2011nonlinear,constantin2012modelling,constantin2013equatorial,r6,r7,r8,r9,r10,r11,r12,r13,r14}.

\subsection{The governing equations}
We choose a frame reference with the origin at a point on the Earth's surface, with the $x$-axis chosen horizontally due east, the $y$-axis horizontally due north and the $z$-axis upwards. Let $z=-d$ be the lower boundary of the layer to which the influence of wind is constrained, and $z=\eta(t,x,y)$ be the free surface of the ocean. In the region $-d \leq z \leq \eta(t,x,y)$, the governing equations in the $f$-plane approximation near the Equator are the Euler equations \cite{gallagher2007influence}:
\begin{equation}
    \begin{cases}
    \begin{aligned}
    &u_{t} + u u_{x} + v u_{y} + w u_{z} + 2\omega w &&= -\frac{1}{\rho} P_x, \\
    &v_{t} + u v_{x} + v v_{y} + w v_{z}             &&= -\frac{1}{\rho} P_y,\\
    &w_{t} + u w_{x} + v w_{y} + w w_{z} - 2\omega u &&= -\frac{1}{\rho} P_z -g.\\
    \end{aligned}
    \end{cases}
\end{equation}
Since the fact that the density of water is constant, the equation of mass conservation takes the form
\begin{align}\label{2.2}
u_x+v_y+w_z=0.
\end{align}
Here $t$ represents time, $(u,v,w)$ is the velocity field, $\omega = 73 \times 10^{-6}~ rad/s$
is the constant rotational speed of the Earth round the polar axis towards the east, $P$ is the pressure, $\rho$ is the constant density of water, and $g=9.8 ~m/s^{2}$ is the constant gravitational acceleration on the surface of the Earth. On the water free surface, the pressure of the fluid corresponds to the atmospheric pressure $P_{atm}$, that is
\begin{align}
P= P_{atm} \quad \text{on  }~ z = \eta(t,x,y).
\end{align}
Since the assumption of impermeability within the flat bed, we enforce the condition of zero flux
\begin{align}
w=0\qquad\textrm{on}~ z=-d.
\end{align}
Moreover, the free surface of wave consists of identical fluid particles at each moment in time, so the associated boundary conditions are
\begin{align}
w=\eta_{t}+u \eta_x+v \eta_y \quad\text{on}~ z=\eta(t,x,y).
\end{align}
In the present paper, we consider two-dimensional flows, moving in the zonal direction along the Equator independent of the $y$-axis and $v\equiv0$ throughout the flow. Since the flow is irrotational, then we have
\begin{align}\label{2.6}
u_z-w_x=0.
\end{align}
Furthermore, since we assume that the flow presents no stagnation points, that is
\begin{align}\label{2.7}
u(x,z)>c.
\end{align}
We consider travelling wave, where the space-time dependence of the free surface $\eta$, of pressure $P$, and of the velocity fluid $(u,w)$ take the form $(x-ct)$ with $c<0$. $|c|$ represents westward propagation speed of the wave surface. Taking advantage of the $(x,t)$-dependence of the form $(x-ct)$ and passing to the moving frame, we can reformulate the governing equations:
\begin{equation}
\begin{cases}\label{2.8}
(u-c)u_{x} + w u_{z} + 2\omega w = - P_{x}~~\text{for} ~~ -d< z < \eta(x),\\
(u-c)w_{x} + w w_{z} - 2\omega u = - P_{z} -g~~\text{for}~~-d< z < \eta(x),\\
u_x+w_z=0\quad\text{for }~~-d< z < \eta(x)),\\
u_z-w_x=0\quad \text{for }~~-d< z < \eta(x),\\
w=(u-c)\eta_x\quad \text{for }z=\eta(x),\\
P=P_{atm}\quad \text{for }z=\eta(x),\\
w=0\quad \text{for }z=-d,
\end{cases}
\end{equation}
where $\rho$ is taken to be 1.

The equatorial Stokes wave is a smooth travelling wave solution to the governing equations $\eqref{2.8}$ for which there exists a period $L$ such that the free surface $\eta$ and the velocity field $(u,w)$ have period $L$ in the $x$-variable. $\eta$, $u$ and $P$ are symmetric about the wave crest. Moreover, the wave free surface is strictly monotonic between successive wave crests and wave troughs.

Without loss of generality, we assume the wave crest to be located at $x=0$, then $x=\pm L/2$ is trough lines and $\eta(x)=\eta(-x)$ for all $x\in\mathbb R$ with $\eta$ increasing on $[-L/2,0]$ and decreasing on $(0,L/2]$ (See Figure \ref{f1}).

\subsection{Stream function, velocity potential and  hodograph transform}

We introduce the \emph{stream function} $\psi(x,z)$, defined up to an additive constant by
\begin{align}\label{2.9}
\psi_z=u-c,\qquad \psi_x=-w,
\end{align}
then $\eqref{2.8}$ can be rewritten as
\begin{equation}\label{2.10}
\begin{cases}
\psi_z \psi_{xz} - \psi_x \psi_{zz} -2\omega \psi_x = - P_{x} \quad \text{for }~~-d< z < \eta(x),\\
-\psi_z \psi_{xx} + \psi_x \psi_{xz} - 2\omega (\psi_z+c) = - P_{z} -g\quad \text{for }~~-d< z < \eta(x),\\
u_x+w_z=0\quad\text{for }~~-d< z < \eta(x),\\
u_z-w_x=0\quad \text{for }~~-d< z < \eta(x),\\
\psi=0\quad \text{for }z=\eta(x),\\
\psi=m\quad \text{for }z=-d,\\
P=P_{atm}\quad \text{for }z=\eta(x),\\
u>0\quad \text{for }(x,z)\in{D},
\end{cases}
\end{equation}
with
\begin{equation*}
D=\{(x,z)\in{\mathbb R}:-d\leq z \leq \eta(x)\},
\end{equation*}
that is the mass flux passing to the uniform horizontal flow at speed $c$, which is given by \cite{buhler2014waves}\newcommand{\ud}{\mathrm{d}}
\begin{align}
m=\int_{-d}^{\eta(x)}(c-u)\, \ud z.
\end{align}
From $\eqref{2.10}$, we can infer that $\psi$ is harmonic throughout the fluid domain $D$. Applying the Maximum Principle \cite{stein2010complex}, $\psi$ reaches its maximum and minimum at the boundary. In accordance with $\eqref{2.7}$-$\eqref{2.9}$, it ensures that $\psi$ has maximum on $z=\eta(x)$ and minimum on $z=-d$.
For $z\in [-d,\eta(L/2)]$, we have
\begin{align}
\psi(\frac{L}{2},z)-\psi(-\frac{L}{2},z)=\int_{-L/2}^{L/2} \psi_x\, \ud x=-\int_{-L/2}^{L/2}w\, \ud x.\nonumber
\end{align}
Since $u$ is symmetric about the wave profile, we can deduce
\begin{align*}
\frac{\partial}{\partial z}\left[-\int_{-L/2}^{L/2}w\right]=\int_{-L/2}^{L/2} u_x\, \ud x
=u(\frac{L}{2},z)-u(-\frac{L}{2},z)=0
\end{align*}
and
\begin{align*}
\int_{-L/2}^{L/2}w(x,-d)\, \ud x=0.
\end{align*}
We now show  that stream function $\psi$ inherits the symmetry and periodicity of $u$ and $w$ in the $x$-variable, that is
\begin{align*}
\psi(x,z)=\psi(-x,z),\quad \psi(x+L,z)=\psi(x,z),
\end{align*}
where $(x,z)\in D$.
The function $\alpha(x,z)=\psi(x+L,z)-\psi(x,z)$ is harmonic in the interior of $ D$.Since $\psi$ is equal to zero on the free surface $z=\eta(x)$, then $\alpha=0$ on $z=\eta(x)$. Moreover, $\alpha=0$ on $z=-d$. By the Maximum Principle, $\alpha=0$ on $D$.

Note that the function $\beta(x,z)=\psi(x,z)-\psi(-x,z)$ is harmonic in the interior of $ D$. Due to the periodicity of $\psi$, we restrict our analysis to a periodicity box
\begin{align*}
D_L=\{(x,z):-L/2 \leq x \leq L/2,-d\leq z\leq \eta(x)\}
\end{align*}
and it also follows the maximum of $|\beta|$ cannot be attained on the lateral sides of $D$ unless $\beta$ is constant. Applying the Lagrange Mean-Value theorem to $\beta$, that is
\begin{align*}
\beta(x,z)=2x\psi_x(\xi_{x,z},z)=-2xw(\xi_{x,z},z),\quad \xi_{x,z} \in(-x,x).
\end{align*}
If $z=-d$, then $w=0$, and $\beta=0$. $\eqref{2.7}$-$\eqref{2.8}$ together with the monotonicity of $\eta$, we deduce that
\begin{align}\label{2.12}
w(x,\cdot)\geq0~\text{for  }~ x\in [-L/2,0] ~~\text{and}~~ w(x,\cdot) \leq 0~\text{for  }~x\in [0,L/2].
\end{align}
Therefore, $\beta \geq 0$ for all $x\in [-L/2,L/2]$.
Moreover, due to the smoothness of $\psi$, we have
\begin{align*}
\lim_{x \rightarrow L/2} \beta(x,z)=\lim_{x \rightarrow L/2} (\psi(x,z)-\psi(-x,z))=0,
\end{align*}
which implies $\beta(L/2,z)=0$. Due to the symmetry of $\eta$ and $u$ about the wave profile,
we have
\begin{align*}
w(x,\eta(x))=(u(x,\eta(x))-c)\eta_x(x)=-(u(-x,\eta(x))-c)\eta_x(-x)=-w(-x,\eta(x)).
\end{align*}
Therefore, we obtain
\begin{align*}
\frac{\partial}{\partial x}\beta(x,\eta(x))&=\psi_x(x,\eta(x))+\psi_x(-x,\eta(x))\\
&=-w(x,\eta(x))-w(-x,\eta(x))\\
&=0,
\end{align*}
then $\beta=0$ on $z=\eta(x)$. Taking advantage of the Maximum Principle, we obtain $\beta=0$ throughout the fluid domain ${D}$.
For the symmetry of $u$, $w$ and $\psi$, that is
\begin{equation}\label{P}
\begin{cases}
u(x,z)=u(-x,z),\\
w(x,z)=-w(-x,z),\\
\psi(x,z)=\psi(-x,z),
\end{cases}
\end{equation}
which shows that $w(0, \cdot)=0$ and since $w$ is periodic in $x$ variable, we have $w(\pm L/2, z)=0$.

Now we define the \emph{velocity potential} $\varphi(x, z)$ as the harmonic conjugate of $\psi(x,z)$, defined uniquely up to an additive constant by
\begin{align}\label{2.14}
\varphi_x=\psi_z,\quad\varphi_z=-\psi_x\,\quad \text { on } -d\leq z \leq \eta(x).
\end{align}
From $\eqref{2.7}$, $\eqref{2.9}$, $\eqref{2.12}$ and $\eqref{2.14}$, we can deduce that $\varphi$ is constant and equal to its minimum value $\varphi_{\min}$ on $x=-L/2$, and constant equal to its maximum value $\varphi_{max}$ on $x=L/2$. The fact that the function $\psi$ attain its minimum at the flat bed $z=-d$, and attain its maximum at the free surface $z=\eta(x)$. Without loss of generality, we assume that $\varphi_{min}=\psi_{min}=0$, then
\begin{align*}
\varphi_{max}=\int_{-L/2}^{L/2}\varphi_x \,\ud x=\int_{-L/2}^{L/2}(u(x,z)-c)\,\ud x
\end{align*}
and
\begin{align*}
\psi_{max}=\int_{-d}^{\eta(x)}\psi_z\, \ud z &=-\int_{-d}^{\eta(x)}(c-u(x,z))\ud z\\
&=-m=|m|,
\end{align*}
where $m<0$.
The function $f=\varphi + i \psi$, called the \emph{hodograph transform}, is analytic in the interior of $D$. In light of $\eqref{2.2}$ and $\eqref{2.8}$, $f'=u-c+i(-w)$ is also analytic in the interior of $D$. $f$ is a biholomorphic map between the interior of the periodicity region $D_L$ and the closed rectangle
\begin{align*}
R_{0}=\{(q,p):0\leq q \leq \varphi_{max},0\leq p \leq|m|\},
\end{align*}
which extends to a homeomorphism between the closures of these domains. The hodograph transform:
\begin{equation}\label{2.15}
\begin{cases}
q=\varphi(x,z),\\
p=\psi(x,z),
\end{cases}
\end{equation}
which is a global diffeomorphism from the fluid domain $D$ to the rectangle domain $\mathcal R$ with
\begin{align*}
\mathcal R=\{(q,p)\in\mathbb{R}^2: 0 \leq p \leq |m|\}.
\end{align*}

\begin{figure}[H]
\centering
\begin{tikzpicture}
    \draw[->, dashed] (-3,0) -- (4,0) node[right] {$x$}; 
    \draw[->, dashed] (0,-2) -- (0,2) node[above] {$z$}; 

    \draw (-3,-2) -- (3,-2);

    \draw[domain=-3:3,smooth,variable=\x,black,samples=100] plot ({\x},{(exp(-abs(\x))*cos(deg(-abs(\x))))-0.2});

    \draw (-3,-2) --(-3,{(exp(-abs(-3))*cos(deg(-abs(-3))))-0.2});
    \draw (3,-2) -- (3,{(exp(-abs(3))*cos(deg(-abs(3))))-0.2});
    \node at (1.5,-0.4) {$z = \eta(x)$};
    \node at (1.5,-1.8) {$z=-d$};
    \node at (-3,-2.4) {$x=-\frac{L}{2}$};
    \node at (3,-2.4) {$x=\frac{L}{2}$};
    \node at (0,-2.4) {$x=0$};

    \draw[->] (3.5,-1) -- (5.7,-1);
    \draw (6,-0.3) rectangle (10,-2);
    \node at (6,-2.2) {$q=0$};
    \node at (10,-2.2) {$q=\varphi_{max}$};
    \node at (8,-0.5) {$p=|m|$};
    \node at (8,-1.8) {$p=0$};
    \node at (4.6,-0.8) {$q=\varphi(x,z)$};
    \node at (4.6,-1.3) {$p=\psi(x,z)$};
\end{tikzpicture}
\caption{The conformal hodograph transform $f$ maps the fluid domain in the moving frame into a rectangle domain in the $\left(q, p\right)$-plane.}\label{f1}
\end{figure}
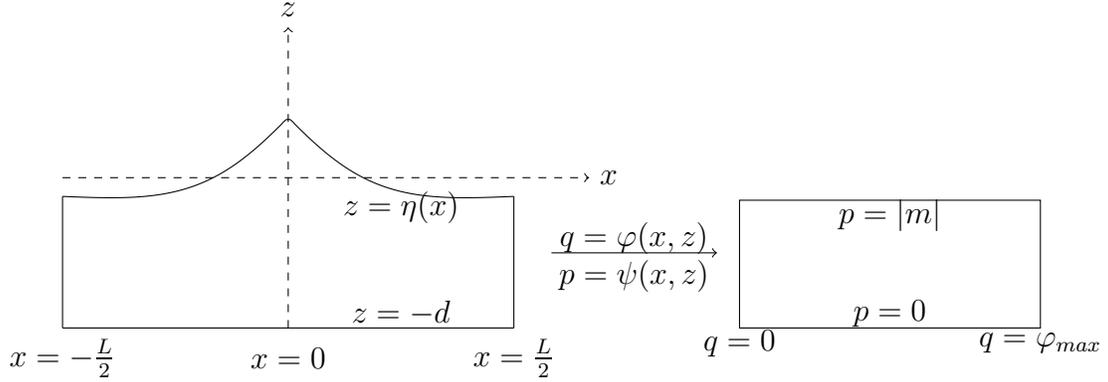

We can observe that the conformal bijection $q+ip\longrightarrow x+iz$ turns the rectangle domain into the fluid domain passing to the moving frame, and there is an inverse transform
\begin{equation}
\begin{cases}
x=x(q,p),\\
z=z(q,p),
\end{cases}
\end{equation}
which has the ``periodicity'' properties(see Figure \ref{f1})
\begin{equation}\label{2.17}
\begin{cases}
x(q+\varphi_{max},p)=x(q,p)+L\quad \text{for} \quad(p,q)\in \mathcal R,\\
z(q+\varphi_{max},p)=z(q,p)\quad \text{for} \quad (p,q)\in \mathcal R,\\
z(q,m)=-d\quad \text{for}\quad q\in\mathbb R.
\end{cases}
\end{equation}

\subsection{Outline of the paper}

The aim of this paper is to investigate kinetic energy and streamline properties  for irrotational equatorial wind waves.
Making use of  complex analysis approach, we derive several properties of streamlines and kinetic energy characteristics which include the streamline time-period, the length of streamlines, the area within the region between two streamlines, and the total energy within the region bounded by two streamlines. Moreover, we analyze the total energy over a streamline time-period, as well as the extremum of the kinetic energy. For example, we show that the kinetic energy over a streamline time-period decreases as the depth increases, and its logarithm is a concave function of the depth. The outline of this paper is as follows.


In Section \ref{sec3}, we investigate the integral means of the kinetic energy. An important step in our method is to notice that the integral of kinetic energy can be expressed as an integral over a horizontal segment of the modulus of a appropriately defined analytic function,  which allows us to apply the Hardy's convexity  theorem. Therefore, we prove the convexity and logarithmic convexity properties of the integral of the kinetic energy, as well as its non-increasing nature in the rectangle domain. In Ref.\cite{constantin2022complexa}, kinetic energy index $s\in[1,+\infty)$. In this section, we extend kinetic energy index $s\in (-\infty,-\frac{1}{2}]\cup [\frac{1}{2},\infty)$, which is not only of great significance to investigate generalized kinetic energy distribution and marine energy development and utilization, but also can be used to predict changes in kinetic energy well by measuring velocities at the free surface and along the flat bed in physics. Furthermore, we show that kinetic energy passing to the moving frame within the whole fluid domain is equal to half of area of the rectangle domain.

In Section \ref{sec4}, we derive some qualitative results about the streamline time-period of a fluid particle that refers to the time required to repeat the same trajectory.  We derive that the streamline time-period is independent of initial data by ideas in Ref.\cite{constantin2022complex}. We extend Ref.\cite{constantin2022complex} from being independent of initial data to being independent of any point on the streamline and any moment, which only depends on the streamline. By means of the Euler-Lagrange method, we prove the convexity and logarithmic convexity properties of the time-period of a fluid particle, as well as its non-increasing nature in the rectangle domain. Moreover, we estimate an upper bound of the minimum of time-period.

In Section \ref{sec5}, we obtain some properties of total kinetic energy and total kinetic energy passing to the moving frame for any fluid particle over a time-period. We derive that the total kinetic energy over a streamline time-period is independent of any point on the streamline, then we prove the convexity and logarithmic convexity properties of total kinetic energy over a time-period, as well as its non-increasing nature in the rectangle domain, and total kinetic energy passing to the moving frame is equal to a constant. Moreover, we provide a better understanding of the kinetic energy for wind waves.

In Section \ref{sec6}, we discuss some properties about the kinetic energy and area within the domain between any streamline and the free surface. We show that the total kinetic energy within the domain between any streamline and the free surface is concave and strictly non-decreasing in the rectangle domain, and total kinetic energy passing to the moving frame is strictly non-decreasing, which generalize the results in Ref.\cite{constantin2023complex}. We find that there exists a difference that total kinetic energy passing to the moving frame becomes concave and strictly non-decreasing in certain situation. Moreover, we obtain   area within the domain between any streamline and the free surface is  concave and strictly non-decreasing, and  estimate an upper bound of the maximum of area.

In Section \ref{sec7}, we present some properties about the length of any streamline and extremum of the kinetic energy passing to the moving frame. We prove that the length of any streamline is convex and non-increasing in the rectangle domain, where we estimate an upper bound of the minimum of streamline length. Motivated by the argument in Ref.\cite{constantin2022complexa}, applying the Bernoulli's law for gravity water waves \cite{constantin2011nonlinear}, we find that the extremum of the kinetic energy passing to the moving frame is attained on the free surface. More precisely, the kinetic energy passing to the moving frame attains its minimum at the wave troughs, the maximum of which is attained at the wave crests.

\section{ Integral means of the kinetic energy }\label{sec3}
In this section, we investigate  the convexity of integral means of kinetic energy $E=((u-c)^2 +w^2)/2$ at different depth levels. Since $f(x,z)=(q,p)$, we can use the hodograph transform $f$ to turn $E(x,y)$ in the fluid domain into $E(f^{-1}(q,p))$ in the closed rectangle domain.  Define
\begin{align}\label{3.1}
\mathcal{M}_s(E,p) = \frac{1}{\varphi_{\max}} \int_0^{\varphi_{max}} E(f^{-1}(q,p))^s\,\ud q,
\end{align}
where $s\in(-\infty,-\frac{1}{2}]\cup [\frac{1}{2},+\infty)$.
\begin{theorem}\label{Theorem3.1}
Suppose $p\in[0,\left|m\right|]$ and $s\in(-\infty,-\frac{1}{2}]\cup [\frac{1}{2},+\infty)$. The function $\mathcal{M}_s(E,p)$ is non-increasing and convex, and $\log\mathcal{M}_s(E,p)$ is convex.
\end{theorem}
\begin{proof}
When $s\in [\frac{1}{2}, +\infty)$, suppose that $R_0=[0,\varphi_{max}]\times[0,|m|]$ and $k=2\pi/\varphi_{max}$.  The function $f=q+ip=\varphi+i\psi$ is analytic and continuous in the interior of $D$ by relations of $\eqref{2.2}$ and $\eqref{2.6}$, which maps the fluid field $D$ onto the closed rectangle $\mathcal R$. $f'=\varphi_x+i\psi_x=u-c+i(-w)$ is analytic, and $|f'(f^{-1}(q,p))|^2=2E(f^{-1}(q,p))$. The periodicity of $f'$ facilitates us the shifting of our research focus from the fluid field $D_L$ to the closed rectangle region $R_0$. Considering the conformal diffeomorphism:
$q+ip \mapsto \mathrm{e}^{ik(q+ip)},$
which maps the periodic rectangle box $R_0$ onto an annulus
\begin{align*}
\mathcal{S} = \left\{\xi \in \mathbb{C}: \mathrm{e}^{km} \leq |\xi| \leq 1\right\}.
\end{align*}
\begin{figure}[H]
\centering
\begin{tikzpicture}
    \draw[line width=2pt, red] (-10,-2) -- (-6,-2); 
    \draw[line width=2pt, blue] (-10,2) -- (-6,2); 
    \draw (-10,-2) -- (-10,2); 
    \draw (-6,-2) -- (-6,2); 
    \draw[line width=2pt, brown] (-10,0) -- (-6,0);
    \node at (-10,-2.2) {$q=0$};
    \node at (-6,-2.2) {$q=\varphi_{\max}$};
    \node at (-8, 1.8) {$p=|m|$};
    \node at (-8,-1.8) {$p=0$};
    \node at (-8,0.2) {$p=p_0$};
    \node at (-5.5,-3) {$q+ip$};
    \node at (-3,-2.93) {$\mathrm{e}^{ik(q+ip)}$};
    \draw[green,ultra thick,dashed] (-2,0) -- (-1,0);
    \draw [blue,ultra thick] (0,0) circle (1);
    \draw [red, ultra thick](0,0) circle (2);
    \draw [brown, ultra thick](0,0) circle (1.5);
    \draw[<->,blue,ultra thick,dashed] (-8,2) to[out=0,in=135] (0.8,0.6);
    \draw[<->,brown,ultra thick,dashed] (-8,0) to[out=0,in=225] (-1.5,0);
    \draw[<->,red,ultra thick,dashed] (-8,-2) to[out=0,in=225] (1.732,-1);
    \draw[|->] (-4.8,-3) -- (-3.7,-3);

\end{tikzpicture}

\caption{The conformal diffeomorphism maps the rectangle domain onto the annulus.}
\end{figure}
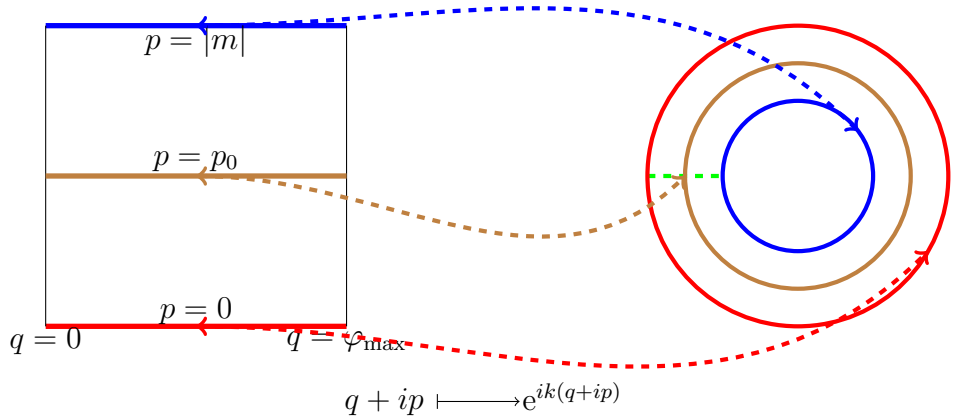

Note that there exists an inverse map $\xi\mapsto-\frac{i}{k}\log \xi$ from $S$ to $\mathcal R$. Consider that the map $F$: $\mathcal S \rightarrow \mathbb {C}$, given by
\begin{align*}
F(\xi)=(f'\circ f^{-1})(-\frac{i}{k}\log \xi),
\end{align*}
which can extend to a smooth function on $\mathcal{S}$. Applying the Morera's theorem, and the analyticity of $f$ and $f'$, the function $F(\xi)$ is  analytic in the interior of $\mathcal S$.
Let $F\in H^{2s}(\mathcal S)$, where $H^{2s}(\mathcal S)$ denotes the Hardy space, then we have
\begin{align*}
M_{2s}(F,r) &= \frac{1}{2\pi}\int_0^{2\pi} |F(r\mathrm{e}^{i\theta})|^{2s} \,\mathrm{d}\theta  \\
&= \frac{1}{2\pi}\int_0^{2\pi} |(f'\circ f^{-1})\left(-\frac{i}{k}\log(r\mathrm{e}^{i\theta})\right)|^{2s} \,\ud \theta  \\
&= \frac{1}{2\pi}\int_0^{2\pi} \left|(f'\circ f^{-1})\left(\frac{\theta}{k}-i\frac{\log r}{k}\right)\right|^{2s} \,\ud\theta  \\
&= \frac{1}{2\pi}\int_0^{2\pi} \left|(f'\circ f^{-1})\left(\frac{\theta}{k},-\frac{\log r}{k}\right)\right|^{2s} \,\ud \theta.
\end{align*}
Let $\gamma=\frac{\theta}{k}$, we have
\begin{align*}
M_{2s}(F,r)&=\frac{k}{2\pi}\int_0^{\frac{2\pi}{k}} \left|f'\left(f^{-1}(\gamma,-\frac{\log r}{k})\right)\right|^{2s} \,\ud \gamma\\
&=\frac{2^s}{\varphi_{max}}\int_0^{\varphi_{max}}\left |E\left(f^{-1}(\gamma,-\frac{\log r}{k})\right)\right|^s \,\ud \gamma\\
&=2^s\mathcal M_s(E,-\frac{\log r}{k}),
\end{align*}
which implies
\begin{align}\label{3.2}
\mathcal M_s(E,p)=\frac{1}{2^s} M_{2s}(F,\mathrm{e}^{-kp}),
\end{align}
and
\begin{equation}\label{NA}
\log \mathcal M_s(E,p)=\log M_{2s}(G,\mathrm{e}^{-kp})-s\log2.
\end{equation}
Applying the Hardy's convexity theorem for analytic function in an annulus \cite{sarason1965h}, for $2s\geq1$ and
$r\in\left[\mathrm{e}^{km},1\right]$, $M_{2s}(F,r)$ is a non-decreasing function of $r$, and $\log M_{2s}(F,r)$ is a convex function of $\log r$, i.e. for any $r_1$, $r_2$ and every $\alpha \in [0,1]$, we have
\begin{align*}
M_{2s}(F,r) \leq [M_{2s}(F,r_1)]^{\alpha} [M_{2s}(F,r_2)]^{1-\alpha},
\end{align*}
where $\log r =\alpha\log r_1 +(1-\alpha)\log r_2$. By $\eqref{NA}$, the function $\log\mathcal M_s(E,p)$ is convex. We let
\begin{align}
\delta(r)=\ln M_{2s}(F,r)\Rightarrow M_{2s}(F,r)=\mathrm{e}^{\delta(r)}.
\end{align}
Since $\delta(r)$ is a convex function, we have
\begin{align}\label{3.4}
\frac{\ud^2}{\ud r^2}M_{2s}(F,r)=\mathrm{e}^{\delta(r)}(\delta'^{2}(r)+\delta''(r))\geq 0.
\end{align}
Therefore, the convexity of the map $r \mapsto \log M_{2s}(F,r)$ naturally implies the convexity of the map $r \mapsto M_{2s}(F,r)$. Since the map $r\mapsto M_{2s}(F,r)$ is non-decreasing, and the map $p\mapsto\mathrm{e}^{-kp}$ is non-increasing, then the map $p\mapsto M_{2s}(F,\mathrm{e}^{-kp})$ is a non-increasing function. Since $F$ is analytic on the annulus $\mathcal S$, then
\begin{equation*}
M_{2s}(F,r)=\sum_{-\infty}^{\infty}|c_n|^{2s}r^{2sn}, \quad r\in[\mathrm{e}^{km},1],
\end{equation*}
where $F(\xi)=\sum_{-\infty}^{\infty}c_n \xi^n$.
Setting $r=\mathrm{e}^{-kp}$, we obtain
\begin{equation*}
M_{2s}(F,\mathrm{e}^{-kp})=\sum_{-\infty}^{\infty}|c_n|^{2s}\mathrm{e}^{-2skpn}, \quad p\in[0,|m|].
\end{equation*}
For $p\in[0,|m|]$, we have
\begin{align*}
\frac{\ud^2}{\ud p^2}M_{2s}(F,\mathrm{e}^{-kp})=4s^2k^2\sum_{-\infty}^{\infty}n^2|c_n|^{2s}\mathrm{e}^{-2skpn}\geq 0.
\end{align*}
Therefore, the convexity of $r \mapsto \ M_{2s}(F,r)$ implies the convexity of $p \mapsto M_{2s}(F,\mathrm{e}^{-kp})$.
In light of $\eqref{3.2}$, we can infer that the function $\mathcal{M}_{s}(E,p)$ is non-increasing and convex, for $ p\in[0,|m|]$ and any $s\in\left[\frac{1}{2},+\infty\right)$. In light of $\eqref{2.7}$, we have $f'\neq0$ throughout the flow.

When $s\in(-\infty,-\frac{1}{2}]$, then
\begin{align}
\mathcal{M}_s(E,p)&= \frac{1}{\varphi_{\max}} \int_0^{\varphi_{max}} \left( E(f^{-1}(q,p))\right)^s\,\ud q \nonumber \\
&=\frac{1}{\varphi_{\max}} \int_0^{\varphi_{max}} \left(\frac{1}{E(f^{-1}(q,p))}\right)^{-s}\,\ud q \nonumber \\
&=\frac{1}{2^s\varphi_{\max}} \int_0^{\varphi_{max}} \left(\frac{1}{(u(f^{-1}(q,p))-c)^2+(w(f^{-1}(q,p)))^2}\right)^{-2s}\,\ud q \nonumber\\
&=\frac{1}{2^s\varphi_{\max}} \int_0^{\varphi_{max}} \left|\frac{1}{f'(f^{-1}(q,p))}\right|^{-2s}\,\ud q \nonumber\\
&=\frac{1}{2^s\varphi_{\max}} \int_0^{\varphi_{max}} |(f^{-1})'(q,p))|^{-2s}\,\ud q.
\end{align}
Since $f$ is a bijection and $f\neq0$, the function $G(q,p)=(f^{-1})'(q,p)$
is analytic in the interior of the rectangle domain $\mathcal R$. Consider that the map: $\mathcal S\rightarrow \mathbb C$, given by
\begin{equation*}
\hat{G}(\xi)=G(-\frac{i}{k}\log\xi).
\end{equation*}
According to the Morera's theorem, $\hat{G}(\xi)$ is analytic on $\mathcal S$.
Furthermore, setting $\theta=kq$, then we have
\begin{align}\label{M}
\mathcal{M}_s(E,-\frac{\log r}{k})&=\frac{1}{2^s\varphi_{\max}} \int_0^{\varphi_{max}} |G(q,-\frac{\log r}{k}))|^{-2s}\,\ud q\nonumber\\
&=\frac{1}{2^sk\varphi_{\max}}\int_0^{2\pi}|G(\frac{\theta}{k},-\frac{\log r}{k})|^{-2s}\,\ud \theta\nonumber\\
&=\frac{1}{2^sk\varphi_{\max}}\int_0^{2\pi}|G(\frac{\theta}{k}-i\frac{\log r}{k})|^{-2s}\,\ud \theta\nonumber\\
&=\frac{1}{2^sk\varphi_{\max}}\int_0^{2\pi}|G(-\frac{i}{k}\log(r\mathrm{e}^{i\theta}))|^{-2s}\,\ud \theta\nonumber\\
&=\frac{1}{2^s}{M}_{-2s}(\hat{G},r).
\end{align}
Let $r=\mathrm{e}^{-kp}\in \mathcal S$, we find that
\begin{equation*}
\mathcal{M}_s(E,p)=\frac{1}{2^s}{M}_{-2s}(\hat{G},\mathrm{e}^{-kp}).
\end{equation*}
The proof of properties in $\eqref{M}$ is analogous to that in $\eqref{3.2}$. Therefore, applying the Hardy's convexity theorem, for $-2s\geq 1$, $\log {M}_{-2s}(\hat{G},p)$ is a convex function, and $\log\mathcal{M}_s(E,p)$ is convex. Similar to $\eqref{3.4}$, we infer that the function $\mathcal{M}_s(E,p)$ is convex, which means
\begin{align}\label{3.7}
\frac{\ud}{\ud \, p}\mathcal{M}_s(E,p)&=\frac{1}{2^s\varphi_{\max}} \int_0^{\varphi_{\max}} \frac{\ud}{\ud \,p}|G(q,p))|^{-2s}\,\ud q\nonumber\\
&=\frac{1}{2^s\varphi_{\max}} \int_0^{\varphi_{\max}} \frac{\ud}{\ud \,p}\left(\frac{1}{(u(f^{-1}(q,p)) - c)^2 + (w(f^{-1}(q,p)))^2 }\right)^{-2s} \,\ud q\nonumber\\.
\end{align}
is increasing for $p\in[0,|m|]$.
Let $f^{-1}(q,p)=(x,z)$, according to the chain rule, we obtain
\begin{align}
\frac{\partial}{\partial p}\{(u-c)^2+w^2\}=2(u-c)(u_x\frac{\partial x}{\partial p}+u_z\frac{\partial z}{\partial p})+2w(w_x\frac{\partial x}{\partial p}+w_z\frac{\partial z}{\partial p}).
\end{align}
It's easy to find that $\mathcal{M}_s(E,p)=\mathcal{M}_s(E,-p)$.
When $p=-|m|=m$ with $m<0$, by $\eqref{2.17}$, we have
\begin{align}
(x(q,m),z(q,m))=(x(q,m),-d)
\end{align}
and
\begin{align}
w=0\text{, } u_z(\cdot,-d)=w_x(\cdot,-d)\,\quad\text{on }z=-d.
\end{align}
Furthermore, computing the inverse of the Jacobian matrix of $f$, we have
\begin{align}\label{3.11}
\frac{\partial x}{\partial p}(q, m)=-\frac{w}{(u-c)^2+w^2}(x(q,m),-d)=0.
\end{align}
In light of $\eqref{3.7}$-$\eqref{3.11}$, for $p\in[m,0]$, we have
\begin{align}
\mathcal{M}_s'(E,p)\geq 0.
\end{align}
Since $\mathcal{M}_s(E,p)$ is an even function, for $-p\in[m,0]$, we have
\begin{align}
\mathcal{M}_s'(E,p)=-\mathcal{M}_s'(E,-p)\leq 0\quad \text{for } p\in[0, |m|],
\end{align}
where $m<0$. Therefore, the function $\mathcal{M}_s(E,p)$ is non-increasing for $p\in[0, |m|]$, which we complete the proof of the theorem.
\end{proof}

Theorem $\ref{Theorem3.1}$ provides a practical motivation for a better understanding of the structure of the kinetic energy of equatorial wind waves. Indeed, the velocity measurement in physics always is performed on the ocean level and the flat bed. Our consideration can be used to predict the variation of kinetic energy of the fluid particles in the fluid domain. Moreover, for equatorial wind waves, the convexity and logarithmic convexity properties of integral of kinetic energy provide theoretical foundations for numerical simulations of kinetic energy.

\begin{theorem}\label{Theorem3.2}
The kinetic energy within the fluid domain $D$, given by
\begin{align*}
E_0=\frac{u^2+w^2}{2},
\end{align*}
and the kinetic energy within the fluid domain $D$ passing to the moving frame, given by
\begin{align*}
E=\frac{(u-c)^2+w^2}{2},
\end{align*}
satisfy
\begin{align}
\iint_D E_0(x,z) \,\ud A(x,z)=\frac{1}{2}\iint_{\mathcal R }\frac{E_0}{E}(f^{-1}(q,p)) \,\ud A'(q,p)
\end{align}
and
\begin{align}
\iint_D E(x,z) \,\ud A(x,z)=\frac{1}{2}|\mathcal R|,
\end{align}
where
\begin{align*}
\mathcal R=\{(q,p): q\in \mathbb{R},0\leq p \leq|m|\}
\end{align*}
is a rectangle domain.
Furthermore, if we consider periodic domain $D_L$ and $R_0$, then we  obtain
\begin{align}
\iint_{D_{L}} E(x,z) \,\ud A(x,z)=\frac{1}{2}|R_0|=\frac{\varphi_{max}}{2}|m|.
\end{align}
\end{theorem}
\begin{proof}
The hodograph transform $f$: $\mathbb R\rightarrow \mathbb C$, given by $f=q+ip=\varphi+i\psi$. $f'=u-c+i(-w)$ is smooth, $|f'|^2=2E$ and $|f'+c|^2=2E_0$. Moreover, since $f(x,z)=(q,p)$, we have $(x,z)=f^{-1}(q,p)$. The Jacobian of $(x,z)$ is
\begin{align}\label{3.17}
|J|=
\begin{vmatrix}
\varphi_x & \varphi_z\\
\psi_x    & \psi_z\\
\end{vmatrix}
=
\begin{vmatrix}
u-c  & w\\
-w   & u-c\\
\end{vmatrix}
=|f'|^2.
\end{align}
Therefore, we have
\begin{align*}
\iint_D E_0(x,z) \,\ud A(x,z)&=\frac{1}{2}\iint_D |f'+c|^2 \,\ud A(x,z)\\
&=\frac{1}{2}\iint_{\mathcal R} |f'+c|^2 \frac{1}{|f'|^2}(f^{-1}(q,p)) \,\ud A'(q,p)\\
&=\frac{1}{2}\iint_{\mathcal R }\frac{E_0}{E}(f^{-1}(q,p)) \,\ud A'(q,p).
\end{align*}
If $c=0$, then $E_0=E$. Therefore, we have
\begin{align*}
\iint_D E(x,z) \,\ud A(x,z)=\frac{1}{2}|\mathcal R|.
\end{align*}
Considering the periodic rectangle domain
\begin{align*}
R_{0}=\{(q,p):0\leq q \leq \varphi_{max},0\leq p \leq|m|\},
\end{align*}
then we have
\begin{align*}
\iint_{D_{L}} E(x,z) \,\ud A(x,z)=|R_0|=\frac{\varphi_{max}}{2}|m|.
\end{align*}
\end{proof}
\begin{remark}
Theorem $\ref{Theorem3.2}$ ensures that passing to the moving frame, the kinetic energy within the fluid domain $D$ is equal to half of area of rectangle domain $\mathcal R$. It is of significance for studying energy transfer and distribution patterns in fluid dynamics.
\end{remark}

\section{The periodicity of the fluid particle's trajectory}\label{sec4}
In this section, we take advantage of the Euler-Lagrange method to derive some properties of flow. The equatorial Stokes wave is a smooth travelling wave solution to the governing equations $\eqref{2.8}$, which implies that the trajectory of any fluid particle is periodic. For the steady flow, the streamline coincides with the trajectory. The transformation $\eqref{2.15}$ shows that for any point $(x,z)$ in the fluid domain $D$ there is a unique point $(q,p)$ correspondence in the closed rectangle domain $\mathcal R$. $f=q+ip$ is a bijection, which means the points in two regions correspond one-to-one. So there is a smooth streamline $z=g_p(x)$ in the fluid domain $D$ only corresponding to a horizontal line $p$ in the rectangle domain $\mathcal R$. Due to the function $\psi$ composed of $x$ and $z$, we have
\begin{align}
\psi(x, g_p(x))=p.
\end{align}
Applying the implicit differentiation rule, from $\eqref{2.14}$ and $\eqref{2.9}$, we have
\begin{align}\label{4.2}
\psi_x+\psi_z g'_p(x)=0\Rightarrow g'_p(x)=-\frac{\psi_x}{\psi_z}=\frac{\varphi_z}{\varphi_x}=\frac{w}{u-c}.
\end{align}
In light of $\eqref{2.12}$, we obtain
\begin{align}
g'_p(x)\geq0~~\text{for}~x\in [-L/2,0] \quad \text{and} \quad g_p'(x) \leq 0~~\text{for}~x\in [0,L/2].
\end{align}
We apply the  Lagrange method to describe the motion of fluid particles. Therefore, any fluid particle follows the following differential systems:
\begin{equation}\label{4.4}
\begin{cases}
\frac{\ud\, x}{\ud\,t}=u(x,z,t)-c,\\
\frac{\ud\, z}{\ud\,t}=w(x,z,t).
\end{cases}
\end{equation}
For $t=t_0$, we have $(x(t_0),z(t_0))=(x_0,z_0)$. Obviously, if $(u,w)=(c,0)$, then it satisfies $\eqref{4.4}$. In this case, the water keep still passing to the moving the frame. From now on, we assume $(u,w)\neq(c,0)$. Due to the boundedness and smoothness of the velocity field, there is a unique global solution $(x_p(t; x_0,t_0), z_p(t; x_0,t_0))$ depending on initial data. In fact, we can observe that the trajectory of fluid particles is periodic in the closed rectangle field $\mathcal R$ after the hodograph transform of the fluid field(see Figure \ref{f1}).
For $(x,z)\in D$ and $L>0$, we have
\begin{align}
g_p(x)=g_p(x-L).
\end{align}
By $\eqref{2.17}$, we have
\begin{align*}
g_p(x(q,p))=g_p(x((q+T),p)-L),\quad x(q,p)=x((q+T),p)-L,
\end{align*}
where $T$ represents the time required to repeat the same trajectory for the same fluid particle in the closed rectangle $\mathcal R$.
Considering the moving frame, we obtain
\begin{align}
x(q,p)=x((q+T),p)-cT-L.
\end{align}
In fact, $T$ is independent of initial data $(t_0,x_0)$(see\cite{constantin2022complex}). In the next section, we will consider whether it is independent of any point $(q,p)$ on the streamline and any moment $t$.
\begin{theorem}\label{Theorem4.1}
Suppose $(q,p)\in[0,\varphi_{max}]\times[0,|m|]$ and $L>0$. Passing to the moving frame, the equation
\begin{align}
x_p(q+T, p)-cT=x_p(q,p)+L
\end{align}
has a unique solution $T(p)>0$ which only depends on the streamline and is independent of any point on the streamline.
Moreover, the function $T(p)$ is convex and non-increasing, and $\log T(p)$ is a convex function.
\end{theorem}
\begin{proof}
Let
\begin{align*}
h_p(t)=x_p(q+t, p)-ct-x_p(q,p)-L,
\end{align*}
the stream function $\psi=p$ and the potential function $\varphi=q$ are independent of time $t$ which means their partial derivative of time is zero, then we have
\begin{align*}
h'_p(t)=u-c\geq\delta,
\end{align*}
where $\delta =\underset{(x,z)\in D}{\min} \{u(x,z)-c\} > 0$. Applying the Lagrange Mean-Value theorem, for some $\theta \in (0,t)$, we find that
\begin{align*}
h_p(t)-h_p(0)=h'_p(\theta)t\Rightarrow h_p(t)=h'_p(\theta)t+L\geq\delta t +L,
\end{align*}
which ensures the existence of $T\in\mathbb{R^+}$ with $h_p(T)=0$.
Let
\begin{align*}
\alpha_p(s)=x_p(q+s, p)-c(q+s),
\end{align*}
then we have
\begin{align*}
\alpha_p'(s)=u(\alpha_p(s),g_p(\alpha_p(s)))-c\text{,}\quad\alpha_p(0)=x(q,p)-cq
\end{align*}
and
\begin{align*}
\alpha_p(T)=x_p(q+T, p)-c(q+T)=x_p(q,p)-cq+L.
\end{align*}
For $s\in[0,T]$, we obtain
\begin{align*}
T=\int_0^T\frac{\alpha_p'(s)}{u(\alpha_p(s),g_p(\alpha_p(s)))-c}\,\ud s =\int_{\alpha_p(0)}^{\alpha_p(T)}\frac{1}{u(\alpha_p(s),g_p(\alpha_p(s)))-c}\,\ud \alpha_p(s).
\end{align*}
Furthermore, due to the periodicity of $u$, we have
\begin{align}
T=\int_{x_(q,p)-cq}^{x_p(q,p)-cq+L}\frac{1}{u(\alpha,g_p(\alpha))-c}\,\ud \alpha=\int_{-\frac{L}{2}}^{\frac{L}{2}}\frac{1}{u(x,g_p(x))-c}\,\ud x.
\end{align}
Note that $T$ only depends on the horizontal line $p$ in the closed rectangle domain $\mathcal R$. In other words, it only depends on the streamline $z=g_p(x)$ rather than depending on any point $(q,p)$ on the streamline and any moment $t$.

Next, $f=q+ip=\varphi+i\psi$ is analytic in the interior of fluid domain $D$, the derivative of $f$ is analytic, and
$f'=u-c+i(-w)$. By means of the hodograph transform $f$, we consider
\begin{align}
q=q_p(x)=\varphi(x, g_p(x)).
\end{align}
According to the implicit differentiation rule, we have
\begin{align}\label{4.10}
q'_p(x)=\varphi_x+ \varphi_z g_p'(x).
\end{align}
In light of $\eqref{2.9}$, $\eqref{2.14}$ and $\eqref{4.2}$, we obtain
\begin{align}\label{4.11}
q'(x)=\frac{\varphi_x^2+\varphi_z^2}{\varphi_x}=\frac{(u-c)^2+w^2}{u-c}.
\end{align}
Since $E=((u-c)^2 +w^2)/2$, we have
\begin{align}\label{4.12}
q'(x)=\frac{2E}{u(x,g_p(x))-c}.
\end{align}
For $x\in[-L/2,L/2]$, then
\begin{align}\label{4.13}
T(p)=\int_{-\frac{L}{2}}^{\frac{L}{2}}\frac{1}{u(x,g_p(x))-c}\,\ud x&=\int_{-\frac{L}{2}}^{\frac{L}{2}}\frac{q'(x)}{2E(x,g_p(x))}\,\ud x=\int_{0}^{\varphi_{max}}\frac{1}{2E(f^{-1}(q,p))}\,\ud q.
\end{align}
According to Theorem $\ref{Theorem3.1}$, when $s=-1$, we have
\begin{align*}
T(p)=\frac{\varphi_{max}}{2}\mathcal M_{-1}(E,p).
\end{align*}
Therefore, $T(p)$ is a convex and non-increasing function, and $\log T(p)$ is a convex function.
Theorem $\ref{Theorem4.1}$ shows that the time-period $T(p)$ of fluid particle's trajectory is only concerned with the properties of the streamline itself.
\end{proof}
\begin{theorem}\label{Remark4.1}
The time-period $T(p)$ of the fluid particle's trajectory attains its maximum at $p=0$, namely the free surface $z=\eta(x)$, the minimum of which is attained at $p=|m|$ and less than $\frac{L}{\delta}$, where $\delta=\underset{(x,z)\in D}\min \{u(x,z)-c\}$.
\end{theorem}
\begin{proof}
For $p\in[0,\varphi_{max}]\times[0,|m|]$, we rewrite $\eqref{4.13}$ as
\begin{align*}
T(p)=\int_{0}^{\varphi_{max}}\frac{1}{2E}(f^{-1}(q,p))\,\ud q.
\end{align*}
From $\eqref{2.10}$, we obtain $p=0$ on the free surface $z=\eta(x)$. According to Theorem $\ref{Theorem4.1}$, $T(p)$ is non-increasing for $p\in[0,|m|]$. Therefore, the maximum of $T(p)$ is attained at $p=0$, namely the free surface $z=\eta(x)$, and the minimum of $T(p)$ is attained at $p=|m|$. In light of $\eqref{2.8}$ and $\eqref{2.10}$, we have $w=0$ and $p=m$ for $z=-d$. Moreover, we find that
\begin{align*}
T(-p)=\int_{0}^{\varphi_{max}}\frac{1}{2E(f^{-1}(q,-p))}\,\ud q&=\int_{0}^{\varphi_{max}}\frac{1}{|f'(f^{-1}(q,-p))|^2}\,\ud q=T(p).
\end{align*}
Therefore, $T(p)$ is an even function.
Let $\delta=\underset{(x,z)\in D}\min \{u(x,z)-c\}$, then we have
\begin{align*}
T(|m|)=T(m)=\int_{0}^{\varphi_{max}}\frac{1}{2E}(f^{-1}(q,m))\,\ud q.
\end{align*}
Since $(x,z)=f^{-1}(q,p)$, by $\eqref{4.13}$, then we have
\begin{align*}
T(|m|)=T(m)=\int_{-\frac{L}{2}}^{\frac{L}{2}}\frac{1}{u(x,-d)-c}\,\ud x\leq\frac{L}{\delta}.
\end{align*}
\end{proof}
\begin{remark}
From Theorem $\ref{Theorem4.1}$, we can deduce that the fluid particle trajectory exhibits the same pattern, which undergoes a horizontal shift after each streamline time-period $T(p)$. More importantly, this pattern is only dependent on the streamline itself. Moreover, the fact that the pattern is repeated after an appropriate shift was established in \cite{O}. Theorem $\ref{Remark4.1}$ shows the location of the extreme points of the streamline time-period $T(p)$, and estimates an upper bound of the minimum of time-period $T(p)$, which reflects the nonlinear characteristics of the equatorial wind wave. Moreover, Theorem $\ref{Theorem4.1}$ and Theorem $\ref{Remark4.1}$ are advantageous for simulating the motion of fluid particles numerically in hydrodynamics.
\end{remark}

\section{The total kinetic energy of fluid particle}\label{sec5}
In this section, we discuss the variation of total kinetic energy over a time-period for any fluid particle.
The total kinetic energy of any fluid particle over a streamline time-period, given by
\begin{align}\label{eq:5.1}
\mathcal E(q, p)=\frac{1}{2}\int_0^{T(p)}[(x_p'(t;x(q,p)))^2+(z_p'(t;x(q,p)))^2]\,\ud t,
\end{align}
which is independent of initial data $(x_0,z_0)$.
The total kinetic energy of any fluid particle passing to the moving frame over a streamline time-period is given by
\begin{align}
\mathbb E(q,p)=\frac{1}{2}\int_0^{T(p)}[(x_p'(t;x(q,p))-c)^2+(z_p'(t;x(q,p)))^2]\,\ud t.
\end{align}
To begin with, we prove the monotonicity and convexity of a composite function.
\begin{lemma}\label{Lemma5.1}
Suppose for any $x\in\mathbb{R}$, the functions $\alpha(x)$ and $\beta(x)\geq 0$. Let $M(x)=\alpha(x)\cdot\beta(x)$. If $\alpha$ and $\beta$ are both convex and non-increasing function, then $M$ is a convex and non-increasing function.
\end{lemma}
\begin{proof}
Assume that $\alpha(x)$ and $\beta(x)\geq0$.
Since $\alpha(x)$ and $\beta(x)$ are both convex and non-increasing function on $\mathbb R$, we have
\begin{align*}
M'(x)=\alpha'(x)\beta(x)+\alpha(x)\beta'(x)\geq0,\\
\end{align*}
and
\begin{align*}
M''(x)=\alpha''(x)\beta(x)+2\alpha'(x)\beta'(x)+\alpha(x)\beta''(x)\geq0.\\
\end{align*}
Therefore, $M(x)$ is a convex and non-increasing function, which we complete the proof of the lemma.
\end{proof}
\begin{theorem}\label{Theorem5.1}
Suppose $t\in[0,T(p)]$ and $(q,p)\in[0,\varphi_{max}]\times[0,|m|]$. $\mathcal E(q,p)$ is a convex and non-increasing function, which only depends on the streamline $p=\psi$. Moreover, $\log \mathcal E(p)$ is a convex function. $\mathbb E(q,p)$ is equal to a constant $\varphi_{max}/2$.
\end{theorem}
\begin{proof}
The hodograph transform $f=q+ip=\varphi+i\psi$ is analytic, then $f'=u-c+i(-w)$. Since the fluid field remains unchanged with time in steady flow, we have
\begin{align}
\frac{\partial \varphi}{\partial t}=\frac{\partial q}{\partial t}=0, \quad \frac{\partial \psi}{\partial t}=\frac{\partial p}{\partial t}=0.
\end{align}
Furthermore, we have
\begin{align*}
\frac{\partial}{\partial t}x(q,p)=x_q\frac{\partial q}{\partial t}+x_p\frac{\partial p}{\partial t}=0.
\end{align*}
For $t\in[0,T(p)]$, we have
\begin{align*}
\mathcal E(q, p)&=\frac{1}{2}\int_0^{T(p)}[(x_p'(t;x(q,p)))^2+(z_p'(t;x(q,p)))^2]\,\ud t\\
&=\frac{1}{2}\int_0^{T(p)}[(x_p(t;x(q,p)-ct)'+c)^2+(z_p'(t;x(q,p)))^2]\,\ud t\\
&=\frac{1}{2}\int_0^{T(p)}|f'(x_p(t;x(q,p))-ct, z_p'(t;x(q,p)))+c|^2\,\ud t\\
&=\frac{1}{2}\int_0^{T(p)}|f'(x_p(t;x(q,p))-ct, g_p'(x_p(t;x(q,p))-ct))+c|^2\,\ud t
\end{align*}
Let $X=x_p(t;x(q,p))-ct$, then we have $\ud X=(u(x_p(t;x(q,p))-ct, z_p(t;x(q,p)))-c)\ud t$. Moreover, $X(0)=x_p(x(q,p))$, and \begin{equation*}
X(T(p))=x_p(T(p);x(q,p))-cT(p)=x_p(x(q,p))+L.
\end{equation*}
Due to the periodicity of $u$ and $f'$, we have
\begin{align}\label{5.4}
\mathcal E(q, p)&=\frac{1}{2}\int_{x_p(x(q,p))}^{x_p(x(q,p))+L}|f'(X,g_p(X))+c|^2 \frac{1}{u(X,g_p(X))-c}\,\ud X\nonumber\\
&=\frac{1}{2}\int_{-\frac{L}{2}}^{\frac{L}{2}}|f'(x,g_p(x))+c|^2 \frac{1}{u(x,g_p(x))-c}\,\ud x,
\end{align}
which shows that $\mathcal E$ only depends on the streamline $p=\psi$, and is dependent of any point $(q,p)$ on the streamline.
We consider the following transform:
\begin{align*}
q=q_p(x)=\varphi(x,g_p(x)) , \quad (q,p)=f(x,z).
\end{align*}
In light of $\eqref{4.10}$ and $\eqref{4.11}$, we have
\begin{align}\label{5.5}
\mathcal E(p)&=\frac{1}{2}\int_{-\frac{L}{2}}^{\frac{L}{2}}|f'(x,g_p(x))+c|^2 \frac{1}{u(x,g_p(x))-c}\,\ud x\nonumber\\
&=\frac{1}{2}\int_{0}^{\varphi_{max}}\left|\frac{f'+c}{f'}\right|^2(f^{-1}(q,p)) \,\ud q\nonumber\\
&=\frac{1}{2}\int_{0}^{\varphi_{max}}|1+c(f^{-1})'(q,p)|^2 \,\ud q.
\end{align}
When $c=0$, we can infer that $\mathbb E(p)=\mathcal E(p)=\varphi_{max}/2$.
In light of $\eqref{4.12}$, we have
\begin{align}\label{5.6}
\mathcal E(p)=\frac{1}{2}\int_{0}^{\varphi_{max}}\frac{E_0}{E}(f^{-1}(q,p)) \,\ud q,
\end{align}
where $E_0=(u^2+w^2)/2$ is kinetic energy for a fluid particle at different depths.
Note that $Q(q,p)=1+c(f^{-1})'(q,p)$ is analytic on $\mathcal R$. Consider that the map $\mathcal S\rightarrow \mathbb C$, given by
\begin{equation*}
\hat{Q}(\xi)=(Q\circ f^{-1})(-\frac{i}{k}\log \xi),
\end{equation*}
where
\begin{align*}
\mathcal{S} = \left\{\xi \in \mathbb{C}: \mathrm{e}^{km} \leq |\xi| \leq 1\right\}.
\end{align*}
According to the Morera's theorem, $\hat{Q}(\xi)$ is analytic on $\mathcal S$.
Let $\theta=kq$ and $r=\mathrm{e}^{-kp}$, where $k=2\pi/\varphi_{max}$, we find
\begin{align*}
\mathcal E(-\frac{\log r}{k}))&=\frac{1}{2k}\int_{0}^{2\pi}\left|Q(f^{-1}(\frac{\theta}{k},-\frac{\log r}{k}))\right|^2 \,\ud \theta\\
&=\frac{1}{2k}\int_{0}^{2\pi}\left|Q(f^{-1}(\frac{\theta}{k}-i\frac{\log r}{k}))\right|^2 \,\ud \theta\\
&=\frac{1}{2k}\int_{0}^{2\pi}\left|Q(f^{-1}(-\frac{i}{k}\log(r\mathrm{e}^{i\theta})))\right|^2 \,\ud \theta\\
&=\frac{\varphi_{max}}{2}M_2(\hat{Q},r).
\end{align*}
Setting $r=\mathrm{e}^{-kp}$, we obtain
\begin{equation*}
\mathcal E(p)=\frac{\varphi_{max}}{2}M_2(\hat{Q},\mathrm{e}^{-kp}).
\end{equation*}
The exact same argument as the one used in the proof of Theorem $\ref{Theorem3.1}$ shows that the function $\mathcal E(p)$ is non-increasing and convex, and $ \log \mathcal E(p)$ is convex, which we complete the proof the theorem.
\end{proof}
\noindent Moreover, by $\eqref{3.1}$ and $\eqref{5.6}$, we have
\begin{align}
\mathcal E(p)=\frac{\varphi_{max}}{2} \mathcal M_1(E_0\cdot E^{-1},p).
\end{align}
Since $c$ is a constant, it cannot affect the monotonicity, concavity and convexity. Therefore, according to Theorem $\ref{Theorem3.1}$, for $p\in[0,|m|]$ and $s=2$, the function $\mathcal M_2(E_0,p)$ is convex and non-increasing, and $\log\mathcal
M_2(E_0,p)$ is convex. For $p\in[0,|m|]$ and $s=-2$, the function $\mathcal M_{-2}(E,p)$ is convex and non-increasing, and $\log \mathcal M_{-2}(E,p)$ is convex.
Let $H(p)=\mathcal M_2(E_0,p)\cdot\mathcal M_{-2}(E,p)$, according to Lemma $\ref{Lemma5.1}$, we obtain the function $H(p)$ is convex and non-increasing. It's easy to find that $\log H(p)$ is convex. Furthermore, taking advantage of the Cauchy-Schwarz inequality, we find
\begin{equation*}
\mathcal E(p)=\frac{\varphi_{max}}{2} \mathcal M_1(E_0\cdot E^{-1},p)\leq \frac{\varphi_{max}}{2}\left(\mathcal M_2(E_0,p)\right)^{\frac{1}{2}} \cdot\left(\mathcal M_{-2}(E,p)\right)^{\frac{1}{2}} =\frac{\varphi_{max}}{2}\sqrt {H(p)}.
\end{equation*}
Furthermore, we have
\begin{equation*}
\log \mathcal E(p)\leq\frac{1}{2}\log {H(p)}+\log{\frac{\varphi_{max}}{2}}.
\end{equation*}
Moreover, the function $\log {H(p)}$ is convex.
\begin{remark}
Suppose $p\in[0,|m|]$, Theorem $\ref{Theorem5.1}$ ensures that $\mathcal E(p)$ is always less than $\mathbb E(p)$, i.e. the total kinetic energy of a fluid particle is always less than the total kinetic energy of a fluid particle passing to the moving frame in fluid domain over a streamline time-period $T(p)$.
\end{remark}
Now, we consider extending Theorem $\ref{Theorem5.1}$ to the general case.
\begin{theorem}\label{Theorem5.2}
Suppose $s\in(-\infty, -\frac{1}{2}]\cup [\frac{1}{2}, +\infty)\cup\{0\}$, $t\in[0,T(p)]$ and $(q,p)\in[0,\varphi_{max}]\times[0,|m|]$. $\mathcal E_s(q,p)$ given by
\begin{align}
\mathcal E_s(q, p)=\frac{1}{2}\int_0^{T(p)}[(x_p'(t;x(q,p))^2+(z_p'(t;x(q,p))^2]^s\,\ud t
\end{align}
is a convex and non-increasing function, and only depends on the streamline $p=\psi$. Moreover, the function $\log\mathcal E_s(q,p)$ is convex. For $s\in (-\infty, \frac{1}{2}]\cup [\frac{3}{2}, +\infty)$, passing to the moving frame,
\begin{align}
\mathbb E_s(q,p)=\frac{1}{2}\int_0^{T(p)}[(x_p'(t;x(q,p))-c)^2+(z_p'(t;x(q,p)))^2]^s\,\ud t
\end{align}
is a convex and non-increasing function, and only depends on the streamline $p=\psi$. Moreover, the function $\log\mathbb E_s(q,p)$ is convex.
\end{theorem}
\begin{proof}
The proof this theorem is similar to Theorem $\ref{Theorem5.1}$. Therefore, we only sketch some of main steps.
In light of $\eqref{5.4}$-$\eqref{5.5}$, we have
\begin{align}
\mathcal E_s(q, p)&=\frac{1}{2}\int_{-\frac{L}{2}}^{\frac{L}{2}}|f'(x,g_p(x))+c|^{2s} \frac{1}{u(x,g_p(x))-c}\,\ud x \nonumber\\
&=\frac{1}{2}\int_{0}^{\varphi_{max}}\frac{|f'+c|^{2s}}{|f'|^2}(f^{-1}(q,p)) \, \ud q.\nonumber
\end{align}
Let $c=0$, we have
\begin{align}
\mathbb E_s(q,p)=\frac{1}{2}\int_{0}^{\varphi_{max}}|f'|^{2(s-1)}(f^{-1}(q,p)) \, \ud q.\nonumber
\end{align}
Moreover, we obtain
\begin{align}\label{5.10}
\mathbb E_s(q,p)=2^{s-2}\int_{0}^{\varphi_{max}}E^{s-1}(f^{-1}(q,p)) \,\ud q
\end{align}
and
\begin{align}\label{5.11}
\mathcal E_s(q,p)=2^{s-2}\int_{0}^{\varphi_{max}}\frac{E_0^s}{E}(f^{-1}(q,p)) \,\ud q,
\end{align}
where $E_0=(u^2+w^2)/2$ and $E=((u-c)^2+w^2)/2$.
From $\eqref{5.10}$ and $\eqref{5.11}$, we can obtain that $\mathcal E_s$  and $\mathbb E_s$ only depend on the streamline $p=\psi$.
According to Theorem $\ref{Theorem3.1}$, for $s-1\in(-\infty,-\frac{1}{2}]\cup[\frac{1}{2},+\infty)$, we have
\begin{align*}
\mathbb E_s(p)=2^{s-2}\varphi_{max} \mathcal M_{s-1}(E,p).
\end{align*}
Therefore, for $p\in[0,|m|]$ and $s\in(-\infty,\frac{1}{2}]\cup[\frac{3}{2},+\infty)$, the function $\mathbb E_s(p)$ is convex and non-increasing, and $\log \mathbb E_s(p)$ is convex.
For $\mathcal E_s(p)$, if $s=0$, then we have
\begin{align*}
\mathcal E_s(p)=\frac{T(p)}{2}.
\end{align*}
According to Theorem $\ref{Theorem4.1}$, for $p\in[0,|m|]$, the function $\mathcal E_s(p)$ is convex and non-increasing, and $\log \mathcal E_s(p)$ is convex.
When $s\in(-\infty,-\frac{1}{2}]\cup[\frac{1}{2},+\infty)$, we have
\begin{align*}
\mathcal E_s(p)=2^{s-2}\varphi_{max} \mathcal M_1(E^{s}\cdot E^{-1},p),
\end{align*}
Similar to argument in Theorem $\ref{Theorem5.1}$, according to Theorem $\ref{Theorem3.1}$ and Lemma $\ref{Lemma5.1}$, we deduce that the function $\mathcal E_s(p)$ is convex and non-increasing, and $\log \mathcal E_s(p)$ is convex.
\end{proof}

\section{Kinetic energy and area between two streamlines}\label{sec6}
In this section, we discuss some properties of the total energy within the region
\begin{align}
\mathcal D_p =\{(x,z)\in \mathbb{R}: -L/2 \leq x \leq L/2,g_p(x)\leq z\leq \eta(x)\}\nonumber
\end{align}
between any streamline $z=g_p(x)$ and the free surface $z=\eta(x)$, which is given by
\begin{align}
\mathcal K(p)=\frac{1}{2}\iint_{D_p} (u^2+w^2)\,\ud S
\end{align}
and the total energy within the region $\mathcal D_p $ passing to the moving frame, which is given by
\begin{align}
\mathbb K(p)=\frac{1}{2}\iint_{D_p} ((u-c)^2+w^2)\,\ud S,
\end{align}
where $\ud S$ is area parameter and $p\in[0,|m|]$.
We consider the hodograph transform $f$: $\mathcal D_p\rightarrow \mathcal R_p$, given by
$f=q+ip_1=\varphi+i\psi$,
where
\begin{align*}
\mathcal R_p=\{(q,p_1): 0 \leq q \leq \varphi_{\max},  0 \leq p_1 \leq p\}.
\end{align*}
\begin{theorem}\label{Theorem6.1}
Suppose $p\in[0,|m|]$. The total kinetic energy $\mathcal K(p)$ within the region $\mathcal D_p $ is a concave and strictly non-decreasing function. The total energy $\mathbb K(p)$ within the region $\mathcal D_p $ passing to the moving frame is strictly non-decreasing and equal to $(\varphi_{\max}p)/2$.
\end{theorem}
\begin{proof}
Recall that the hodograph transform $f=q+ip_1=\varphi+i\psi$ maps $\mathcal D_p$ onto $\mathcal R_p$. $f'=u-c+i(-w)$ is analytic on $\mathcal D_p$, $|f'|^2=(u-c)^2+w^2$ and $|f'+c|^2=u^2+w^2$. Since $f(x,z)=(q,p_{1})$, we have $(q,p_{1})=f^{-1}(x,z)$.  Moreover, we have
\begin{align}\label{6.3}
\mathcal K(p)&=\frac{1}{2}\iint_{D_p} (u^2+w^2)\,\ud S(x,z)\nonumber\\
&=\iint_{D_p} E_0(x,z) \,\ud S(x,z)\nonumber\\
&=\frac{1}{2}\iint_{\mathcal R_p} \frac{E_0}{E}(f^{-1}(q,p_1)) \,\ud S'(q,p_1)\nonumber\\
&=\frac{1}{2}\int_0^p \int_0^{\varphi_{\max}}\frac{E_0}{E}(f^{-1}(q,p_1)) \,\ud q \ud p_1.
\end{align}
In light of $\eqref{5.6}$, we have
\begin{align*}
\frac{\ud}{\ud p}\mathcal K(p)&=\frac{1}{2} \int_0^{\varphi_{\max}}\frac{E_0}{E}(f^{-1}(q,p)) \,\ud q\\
&=\mathcal E(p)>0.
\end{align*}
Therefore, $\mathcal K(p)$ is strictly non-decreasing.
Furthermore, we obtain
\begin{align*}
\frac{\ud^2}{\ud p^2}\mathcal K(p)&=\mathcal E'(p) \leq 0.
\end{align*}
According to Theorem $\ref{Theorem5.1}$, $\mathcal E(p)$ is non-increasing, which implies $\mathcal K(p)$ is concave. If $c=0$, then we have $E_0=E$, and we have
\begin{align*}
\mathbb K(p)=\frac{\varphi_{\max}}{2}p,
\end{align*}
which is a strictly non-decreasing function.
\end{proof}
\begin{remark}
Suppose $p\in[0,|m|]$, Theorem $\ref{Theorem6.1}$ ensures that $\mathcal K(p)$ is always less than $\mathbb K(p)$, i.e. the total kinetic energy within the domain $\mathcal D_p $ is always less than the total kinetic energy within the domain $\mathcal D_p $ passing to the moving frame.
\end{remark}
\begin{theorem}
Suppose $p\in[0,|m|]$, $s\in(-\infty,-\frac{1}{2}]\cup[\frac{1}{2},+\infty)$. The total kinetic energy $\mathcal K_s(p)$ within the domain $\mathcal D_p $, given by
\begin{align}
\mathcal K_s(p)=\frac{1}{2}\iint_{D_p} (u^2+w^2)^s\,\ud S
\end{align}
is a concave and strictly non-decreasing function. For $s\in(-\infty,\frac{1}{2}]\cup[\frac{3}{2},+\infty)$, the total energy $\mathbb K_s(p)$ within the domain $\mathcal D_p $ passing to the moving frame, given by
\begin{align}
\mathbb K_s(p)=\frac{1}{2}\iint_{D_p} \left((u-c)^2+w^2\right)^s\,\ud S,
\end{align}
is a concave and strictly non-decreasing function.
\end{theorem}
\begin{proof}
The proof this theorem is the same as Theorem $\ref{Theorem5.2}$. We only sketch some key steps. Similar to $\eqref{6.3}$, we have
\begin{align*}
\mathcal K_s(p)&=\frac{1}{2}\iint_{D_p} (u^2+w^2)^s\,\ud S \\
&=\int_0^p \int_0^{\varphi_{\max}}\frac{E_0^s}{E}(f^{-1}(q,p_1)) \,\ud q \ud p_1,
\end{align*}
and
\begin{align*}
\frac{\ud}{\ud p}\mathcal K_s(p)&=\int_0^{\varphi_{\max}}\frac{E_0^s}{E}(f^{-1}(q,p)) \,\ud q \\
&=\mathcal E_s(p)>0.
\end{align*}
Moreover, we have
\begin{align}
\frac{\ud^2}{\ud p^2}\mathcal K_s(p)=\mathcal E_s'(p)\leq 0.\nonumber
\end{align}
According to Theorem $\ref{Theorem5.2}$, $\mathcal K_s(p)$ is concave and strictly non-decreasing.
If $c=0$, then we have
\begin{align*}
\mathbb K_s(p)=\int_0^p\int_{0}^{\varphi_{max}}E^{s-1}(f^{-1}(q,p_1)) \,\ud q \ud p_1,
\end{align*}
\begin{align*}
\frac{\ud}{\ud p}\mathbb K_s(p)=\mathbb E_s(p)>0
\end{align*}
and
\begin{align*}
\frac{\ud^2}{\ud p^2}\mathbb K_s(p)=\mathbb E_s'(p)\leq 0.
\end{align*}
According to Theorem $\ref{Theorem5.2}$, $\mathbb K_s(p)$ is concave and strictly non-decreasing.
\end{proof}
Now, we consider the area within the domain $\mathcal D_p $, which is given by
\begin{align}
\mathcal S(p)=\iint_{D_{p}}  \,\ud \mathcal S.
\end{align}
\begin{theorem}\label{Theorem6.2}
The area $\mathcal S(p)$ within the domain $\mathcal D_p $ is a concave and strictly non-decreasing function.
\end{theorem}
\begin{proof}
In light of $\eqref{3.17}$, we obtain
\begin{align}
\mathcal S(p)&=\iint_{D_{p}}  \,\ud \mathcal S\nonumber\\
&=\iint_{\mathcal R_{p}} \frac{1}{|f'(f^{-1}(q,p_1))|^2} \,\ud \mathcal S'(q,p_1)\nonumber\\
&=\iint_{\mathcal R_{p}} \left|(f^{-1})'(f^{-1}(q,p_1))\right|^2 \,\ud \mathcal S'(q,p_1)\nonumber\\
&=\int_0^p \int_0^{\varphi_{\max}}\left|(f^{-1})'(f^{-1}(q,p_1))\right|^2 \,\ud q\,\ud p_1.
\end{align}
Since there are no stagnation points throughout the fluid domain, then we have $f'\neq 0$. The function $G(q,p_1)=(f^{-1})'(f^{-1}(q,p_1))$ is  analytic in the interior of $R_p$. By $\eqref{3.1}$, we have
\begin{align*}
\frac{\ud}{\ud p}\mathcal S(p)&=\int_0^{\varphi_{\max}}\left|(G(q,p_1)\right|^2 \,\ud q \nonumber\\
&=\int_0^{\varphi_{\max}}\frac{1}{2E}(f^{-1}(q,p)) \,\ud q \nonumber\\
&=T(p)>0.
\end{align*}
Therefore, the function $\mathcal S(p)$ is strictly non-decreasing. According to Theorem $\ref{Theorem5.1}$, we can obtain
\begin{align*}
\frac{\ud^2}{\ud p^2}\mathcal S(p)=T'(p)\leq 0.
\end{align*}
Therefore, $\mathcal S(p)$ is a concave function.
\end{proof}
\begin{theorem}
The area $\mathcal S(p)$ within the domain $\mathcal D_p $ attains its minimum that is equal to zero at $p=0$, namely the free surface $\eta$, the maximum of which is attained at $p=|m|$ and less than $\frac{|m|\varphi_{\max}}{2\delta_{0}^2}$, where $\delta_0=\underset{(q,p)\in R_0}\min \{u(f^{-1}(q,p))-c\}>0$.
\end{theorem}
\begin{proof}
For $(q,p)\in[0,\varphi_{max}]\times[0,|m|]$, from $\eqref{4.12}$, we have
\begin{align*}
\mathcal S(p)&=\int_0^p \int_0^{\varphi_{\max}}\left|(f^{-1})'(f^{-1}(q,p_1))\right|^2 \,\ud q\,\ud p_1 \\
&=\int_0^p \int_0^{\varphi_{\max}}\frac{1}{2E}(f^{-1}(q,p_1)) \,\ud q\,\ud p_1.
\end{align*}
From $\eqref{2.10}$, then $p=0$ on the free surface $z=\eta(x)$.  According to Theorem $\ref{Theorem6.2}$, $\mathcal S(p)$ is non-decreasing for $p\in[0,|m|]$. Therefore, the minimum of $\mathcal S(p)$ is attained at $p=0$, which is equal to zero. The maximum of $\mathcal S(p)$ is attained at $p=|m|$. In light of $\eqref{2.8}$ and $\eqref{2.10}$, we have $w=0$ and $p=m$ on $z=-d$. It's easy to find that $\mathcal S(p)=-\mathcal S(-p)$. Let $\delta_0=\underset{(q,p)\in R_0}\min \{u(f^{-1}(q,p))-c\}>0$, then we have
\begin{align*}
\mathcal S(|m|)=-\mathcal S(m)&=-\int_0^m \int_0^{\varphi_{\max}}\frac{1}{2E}(f^{-1}(q,m)) \,\ud q\,\ud p_1\\
&=\int_{0}^{|m|} \int_0^{\varphi_{max}}\frac{1}{2(u-c)^2}(f^{-1}(q,m) \,\ud q\,\ud p_1 \leq\frac{|m|\varphi_{\max}}{2\delta_{0}^2},
\end{align*}
where $m<0$.
\end{proof}

\section{Streamline length and kinetic energy extremum }\label{sec7}
In this section, we discuss some properties of the length of a streamline and the extremum of kinetic energy. Let $x\in[-L/2,L/2]$, we consider the length of any streamline $\psi_p=(x, g_p(x))$, given by
\begin{align*}
\mathcal L(p)=\int_{\psi_p}  \ud s,
\end{align*}
where $\ud s$ is the arclength parameter.
\begin{theorem}\label{Theorem7.1}
Suppose $p\in[0,|m|]$. The function $\mathcal L(p)$ is convex and non-increasing, and $\log \mathcal L(p)$ is convex.
\end{theorem}
\begin{proof}
The hodograph transform $f=q+ip=\varphi +i\psi $ is analytic. $f'=u-c+i(-w)$ is analytic,  and $|f'|^2=2E$. In light of $\eqref{4.2}$ and $\eqref{4.12}$, we have
\begin{align*}
\mathcal L(p)&=\int_{\psi_p}  \ud s\nonumber\\
&=\int_{\frac{-L}{2}}^{\frac{L}{2}} \sqrt{1+(g_p'(x))^2} \,\ud x \nonumber\\
&=\int_0^{\varphi_{\max}} \frac{1}{|f'|}(f^{-1}(q,p))\,\ud q \nonumber \\
&=\frac{\sqrt{2}}{2}\int_0^{\varphi_{\max}} \frac{1}{\sqrt{E}}(f^{-1}(q,p))\,\ud q \nonumber \\
&=\frac{\sqrt{2}\varphi_{\max}}{2}\mathcal M(E^{-\frac{1}{2}},p).
\end{align*}
According to Theorem $\ref{Theorem3.1}$, for $s=-\frac{1}{2}$, $\mathcal L(p)$ is convex and non-increasing, and $\log \mathcal L(p)$ is convex.
\end{proof}
\begin{theorem}
The length $\mathcal L(p)$ of the streamline $\psi_p$ attains its maximum at $p=0$, namely the free surface $\eta$, the minimum of which is attained at $p=|m|$ and less than $\frac{\sqrt{2}\varphi_{\max}}{2\delta_0}$, where $\delta_0=\underset{(q,p)\in R_0}\min \{u(f^{-1}(q,p))-c\}>0$.
\end{theorem}
\begin{proof}
For $(q,p)\in[0,\varphi_{max}]\times[0,|m|]$, then
\begin{align*}
\mathcal L(p)=\frac{\sqrt{2}}{2}\int_0^{\varphi_{\max}} \frac{1}{\sqrt{E}}(f^{-1}(q,p))\,\ud q.
\end{align*}
From $\eqref{2.10}$, then $p=0$ on the free surface $z=\eta(x)$.  According to Theorem $\ref{Theorem7.1}$, $\mathcal L(p)$ is non-increasing for $p\in[0,|m|]$. Therefore, the maximum of $\mathcal L(p)$ is attained at $p=0$, namely the free surface $z=\eta(x)$, and the minimum of $T(p)$ is attained at $p=|m|$. In light of $\eqref{2.8}$ and $\eqref{2.10}$, we obtain $w=0$ and $p=m$ on $z=-d$. It's easy to find that $\mathcal L(p)=\mathcal L(-p)$. Let $\delta_0=\underset{(q,p)\in R_0}\min \{u(f^{-1}(q,p))-c\}>0$, then we have
\begin{align*}
\mathcal L(|m|)=\mathcal L(m)=\frac{\sqrt{2}}{2}\int_0^{\varphi_{\max}} \frac{1}{\sqrt{E}}(f^{-1}(q,m))\,\ud q.
\end{align*}
Since $(x,z)=f^{-1}(q,p)$ and $u>c$, then
\begin{align*}
\mathcal L(|m|)=\mathcal L(m)=\frac{\sqrt{2}}{2}\int_0^{\varphi_{\max}} \frac{1}{u-c}(f^{-1}(q,m))\,\ud q \leq \frac{\sqrt{2}\varphi_{\max}}{2\delta_0}.
\end{align*}
\end{proof}
\begin{lemma}\label{lem1}
The kinetic energy $E$, where $E=((u-c)^2+w^2)/2$, decreases on the free surface $z=\eta(x)$ as the elevation $z$ increases.
\end{lemma}
\begin{proof}
According to the Bernoulli's law for gravity water waves \cite{constantin2011nonlinear}, then
\begin{align*}
E+(g-2\omega c)z=\text{constant},
\end{align*}
where wave speed $c<0$ represents propagation of this wave westward, and $\omega >0$
is the constant rotational speed of the Earth round the polar axis towards the east.
Therefore, the kinetic energy decreases as the elevation $z$ increases on the free surface $z=\eta(x)$.
\end{proof}
\begin{theorem}\label{lem2}
The minimum of the kinetic energy $E$ passing to the moving frame is attained at the wave crests, where the maximum is attained at the wave troughs.
\end{theorem}
\begin{proof}
The hodograph transform $f$: $\mathbb{R}\rightarrow \mathbb{C}$, given by
\begin{align*}
f=q+ip=\varphi+i\psi
\end{align*}
is analytic in the interior of the fluid domain $D$. $f'=u-c+i(-w)$ is also analytic and satisfies $|f'|^2=2E$.
Under the assumption of no stagnation points throughout the flow, we have $f\neq0$. Since the minimum of $|f'|$ corresponds to the maximum of $\frac{1}{|f'|}$, we only need to discuss the maximum of $|f'|$ and $\frac{1}{|f'|}$.

Applying the Maximum Principle to $f'$ and $\frac{1}{f'}$, we can deduce $|f'|$ and $\frac{1}{|f'|}$ attain at the boundary unless $f'$ is a constant. According to the Schwarz reflection principle, the imaginary part of $f'$ vanishes on $z=-d$, which implies the maximum of $|f'|$ and $\frac{1}{|f'|}$ cannot be attained at the flat bed $z=-d$ unless $f'$ is a constant \cite{stein2010complex}. Therefore, the extremum of $|f'|$ is attained on the free surface $z = \eta(x)$.
According to Lemma $\ref{lem1}$, the minimum of kinetic energy passing to the moving frame is attained at the wave crests, while the maximum is attained at the wave troughs.
\end{proof}

\section*{Data availability }
No data was used for the research described in the article.

\section*{Conflict of interest }
The authors do not have any other competing interests to declare.

\end{document}